\documentclass{article}

\usepackage{hyperref}
\hypersetup{
    colorlinks=true,
    linkcolor=blue,
    citecolor=blue,
    urlcolor=blue
}

\usepackage[english]{babel}
\usepackage[utf8]{inputenc}
\usepackage{csquotes}
\usepackage{johd}
\usepackage{graphicx} 
\usepackage{setspace}
\usepackage{amsmath}
\usepackage{amsthm}
\usepackage{amssymb}
\usepackage{xcolor}
\usepackage{bm}
\usepackage{authblk}
\usepackage{float}
\usepackage{array}
\usepackage[backend=biber,style=numeric]{biblatex}
\usepackage{thm-restate}

\newtheorem{theorem}{Theorem}
\newtheorem*{theorem*}{Theorem}
\newtheorem{lemma}{Lemma}[section]
\newtheorem{definition}{Definition}

\newtheorem*{example}{Example}

\newtheorem*{proposition}{Proposition}

\newtheorem*{corollary}{Corollary}
\newtheorem*{keywords}{Keywords}
\newtheorem{remark}{Remark}[section]

\newcommand{\Z}{\mathbb{Z}}

\newcommand{\eeone}{EE^{(1)}}
\newcommand{\eetwo}{EE^{(2)}}
\newcommand{\eethree}{EE^{(3)}}
\renewcommand{\leq}{\leqslant}

\newcommand{\z}{\left}
\newcommand{\y}{\right}

\title{Counting Permutations in $S_{2n}$ and $S_{2n+1}$}

\author{Yuewen Luo}

\affil{\small Department of Mathematics, University of Toronto, \texttt{yuewen.luo@mail.utoronto.ca}}

\date{\today}

\begin{document}

\maketitle

\begin{abstract} 
 Let $\alpha(n)$ denote the number of perfect square permutations in the symmetric group $S_n$. The conjecture $\alpha(2n+1) = (2n+1) \alpha(2n)$, provided by Stanley\cite{problem}, was proved by Blum\cite{Blum} using a generating function. This paper presents a combinatorial proof for this conjecture. At the same time, we demonstrate that all permutations with an even number of even cycles in both $S_{2n}$ and $S_{2n+1}$ can be categorized into three distinct types that correspond to each other.\\
\begin{keywords}
    Permutation, disjoint cycles, bijection, mapping, counting.
\end{keywords}
\end{abstract}

\section{Introduction}
\setstretch{1.5}

Let $n\in \Z^+$, $[n] = \{1,2,...,n\}.$ We know that every permutation in $S_n$ can be written as a product of disjoint cycles. Given $m$ distinct numbers $a_1, a_2, ..., a_m$ in $[n]$, $(a_1, a_2, ... ,a_m)$ denotes the $m$-cycle $ a_1\rightarrow a_2\rightarrow ... \rightarrow a_m \rightarrow a_1$. \\

To begin with, let us recall the notions provided by Stanley\cite[page 13]{problem}. Let $w\in S_n$. We say that $w$ is a \textbf{perfect square permutation} if there exists $u\in S_n$ such that $u^2 = w$. We define $\alpha(n)$ to be the number of perfect square permutations in $S_n$. Some direct enumerations are provided in the table below\cite{oeis}.
\vspace{1em}
\begin{table}[!h]
\centering
\small 
\setlength{\tabcolsep}{8pt} 
\begin{tabular}{|c|c|}
\hline
$n$ & $\alpha(n)$ \\ \hline
2 & 1 \\ \hline
3 & 3 \\ \hline
4 & 12 \\ \hline
5 & 60 \\ \hline
6 & 270 \\ \hline
7 & 1890 \\ \hline
8 & 14280 \\ \hline
9 & 128520 \\ \hline
10 & 1096200 \\ \hline
11 & 12058200 \\ \hline
12 & 139043520 \\ \hline
13 & 1807565760 \\ \hline
14 & 22642139520 \\ \hline
15 & 339632092800 \\ \hline
16 & 5237183952000\\ \hline
17 & 89032127184000\\ \hline
\end{tabular}
\caption{Values for $\alpha(n)$}
\label{tab:values}
\end{table}

\vspace{2em}

By observing Table 1, it seems to be true that $\alpha(2n+1) = (2n+1)\alpha(2n)$ for all $n\in \Z^+$. This suggests an intriguing relationship between the numbers of perfect square permutations in $S_{2n}$ and $S_{2n+1}$. In particular, the ratios of the numbers of perfect square permutations in $S_{2n}$ and $S_{2n+1}$ are the same. In 1973, Blum\cite{Blum} proved this conjecture using a generating function. In this paper, we give a combinatorial proof. 

In order to enhance clarity and ease of understanding, a table summarizing some basic notations used in this paper is provided below. The notations will be used consistently throughout the paper to simplify various expressions. If additional notations are introduced later, they will be explained in context.

\vspace{1em}

\begin{table}[H]
    \centering
    \begin{tabular}{>{\raggedright}p{3cm}|>{\raggedright\arraybackslash}p{8cm}}
        \textbf{Notations} & \textbf{Descriptions} \\\hline
        $\underline{PS}$ & Perfect square permutation(s) \\
        $\underline{NPS}$ & Non-perfect square permutation(s) \\
        $\underline{EE}$ & Permutation(s) with an even number of even cycles \\
        $\underline{OE}$ & Permutation(s) with an odd number of even cycles \\
        $\underline{PS_n}$ & The set of perfect square permutations in $S_n$\\
        $\underline{NPS_n}$ & The set of non-perfect square permutations in $S_n$\\
        $\underline{EE_{n}}$ & The set of permutations in $S_n$ with an even number of even cycles  \\
        $\underline{OE_{n}}$ & The set of permutations in $S_n$ with an odd number of even cycles  
    \end{tabular}
\caption{\label{tab:notations}Basic notations used in this paper.}
\end{table}

\subsection{Basic Properties of Perfect Square Permutations} \label{propertiesPS}
\setstretch{1.5}
\begin{lemma}\label{basic_property}
    Let $n, m\in \Z^+$ and $m\leq n$. Let $u=(a_1, a_2, ... ,a_m)$ be an $m$-cycle in $S_n$. 
    \begin{enumerate}
        \item If $m$ is even, then $u^2$ gets split into two disjoint cycles with the same length, $\frac{m}{2}$.
        \item If $m$ is odd, then $u^2$ remains an $m$-cycle. 
    \end{enumerate}
\end{lemma}
    \begin{proof}
        It is readily checked that for all $k\in \Z^+$,
        \begin{align*}
            (a_1, a_2, ... ,a_{2k})^2&= (a_1, a_3, ..., a_{2k-1})(a_2, a_4, ..., a_{2k}),\\ 
            (a_1, a_2, ... ,a_{2k+1})^2&= (a_1, a_3, ..., a_{2k+1}, a_2, a_4, ..., a_{2k}).
        \end{align*}
    \end{proof}

\begin{remark}\label{oddPS}
    Every odd cycle is a perfect square. 
\end{remark}

\begin{lemma}[Necessary Condition] \label{PSisEE}
Let $n\in \Z^+$. If $w\in PS_n$, then $w\in EE_n$.
\end{lemma}

    \begin{proof}
    Assume $w=u^2$ for some $u\in S_n$. We write $u$ as a product of disjoint cycles $$u = \underset{i\in [k]}{\prod}u_i$$ for some $k\in \Z^+$ and $ k\leq n$.

    By \hyperref[basic_property]{Lemma 1.1}, $\underset{i\in [k]}{\prod}u_i^2$ are still disjoint and even cycles are in pairs.  
    \end{proof}

\begin{remark}\label{PSdependence}
     Since all odd cycles are perfect squares (by \hyperref[oddPS]{Remark 1.1}), the determination of whether a permutation is $PS$ or $NPS$ depends solely on the partition of its even cycles - specifically, the number of even cycles with the same length in $PS$ must be even. Some examples are provided below. 

   \begin{example} [$PS$ and $NPS$ in $EE_n$]
         \begin{align*}
             &(1,2)(3,4)(5)(6)(7)(8) \in PS_8\\
             &(1,2)(3,4,9,10)(5)(6)(7)(8) \in NPS_{10}\\
             &(1,2)(3,4)(5,6,7,8)(9,10,11,12)(13) \in PS_{13}\\ 
             &(1,2,3,4,5,6)(7,8)(9)\in NPS_9.
         \end{align*}
    \end{example}
\end{remark}

\subsection{Permutations with Even number of Even Cycles}

By the necessary condition discussed in \hyperref[PSisEE]{Lemma 1.2}, we may focus only on permutations with an even number of even cycles. A direct bijective proof is tempting, but after considering various partitions, we decide to divide both $EE_{2n}$ and $EE_{2n+1}$ into three different types, and then construct bijections between corresponding types.

Furthermore, now that $ PS_{n} \subset EE_n$, we can analyze the perfect square permutations within each type. We will later show that it is possible to establish bijections between  $[2n+1] \times PS_{2n}$ and $PS_{2n+1}$ for each type. 

\vspace{1em}

\begin{definition}\label{2n_types}
    Let $n\in \Z^+$.
    We divide $EE_{2n}$ into three \textbf{distinct} types: 
    \begin{itemize}
        \item Type $1$, denoted $\eeone_{2n}:= \{w\in EE_{2n}:w$  only has even cycles\}. 
        \item Type $2$, denoted $\eetwo_{2n}:= \{w \in EE_{2n}:w$  only has odd cycles\}.
        \item Type $3$, denoted $\eethree_{2n}:= \{w\in EE_{2n}:w$  has both even and odd cycles\}.
    \end{itemize}
    Note that: 
    $$\z|EE_{2n}\y| = \underset{i\in [3]}{\sum} \z|EE_{2n}^{(i)}\y|.$$
\end{definition}

\vspace{1em}

\begin{definition}\label{2n+1_types}
    Let $n\in \Z^+$.
    We divide $EE_{2n+1}$ into three \textbf{distinct} types:
    \begin{itemize}
        \item Type $1$, denoted $\eeone_{2n+1}:=\{w\in EE_{2n+1}:w$  only has even cycles and one $1$-cycle\}. 
        \item Type $2$, denoted $\eetwo_{2n+1}:=\{w\in EE_{2n+1}:w$ only has odd cycles\}.
        \item Type $3$, denoted $\eethree_{2n+1}:=\{w\in EE_{2n+1}:w$ has both even and odd cycles\}$\backslash \eeone_{2n+1}$.
    \end{itemize}
    Note that: 
     $$\z|EE_{2n+1}\y| = \underset{i\in [3]}{\sum} \z|EE_{2n+1}^{(i)}\y|.$$
\end{definition}

\vspace{1em}

\begin{definition}\label{definition3}
    Let $i\in [3]$ and $n\in \Z^+$.
    Define $PS_n^{(i)}:= EE_n^{(i)} \cap PS_n$.
\end{definition}

\vspace{1em}

We are now ready to state our main results.
\setstretch{1.5}

\newtheorem*{theorem:1}{Theorem \ref{theorem:1}}
\begin{theorem:1}
    Let $n\in \Z^+$ and $i\in [3]$. Then
    $$\mathbf{\z|EE^{(i)}_{2n+1}\y| = (2n+1)\z|EE^{(i)}_{2n}\y|}.$$
\end{theorem:1}

\newtheorem*{theorem:2}{Theorem \ref{theorem:2}}
\begin{theorem:2}
    Let $n\in \Z^+$ and $i\in [3]$. Then
    $$\mathbf{\z|PS^{(i)}_{2n+1}\y| = (2n+1)\z|PS^{(i)}_{2n}\y|}.$$
\end{theorem:2}

\begin{corollary}\label{corollary}
    Let $n\in \Z^+$. Then
    $$\mathbf{\alpha(2n+1) = (2n+1)\alpha(2n)}.$$
\end{corollary}

\section{Adding and Swapping Mapping Method}\label{addingswapping}
\setstretch{1.5}
This paper employs a special mapping method that plays an important role in the proofs of the theorems presented later. In this section, we explain how the Adding and Swapping Mapping Method works.

Let $n\in \Z^+$ and $w\in S_{2n}$ with \textbf{no} $1$-cycle. This mapping method maps $w$ to $2n+1$ \textbf{unique} permutations in $S_{2n+1}$.

Define: $\begin{cases}
D_{2n+1}, & \text{by adding } (2n+1) \text{ outside as an 1-cycle, i.e. the identity map}.\\
D_i \ (1 \leq i \leq 2n), & \text{by adding } (2n+1) \text{ outside as an 1-cycle and then swapping the position } \\
& \text{of } (2n+1) \text{ with } i.\\
\end{cases}$

\begin{example}
    Take $w=(1,2)(3,4,5,6)\in S_6$. Then 
    \begin{align*}
        D_1(w)&=(7,2)(3,4,5,6)(1)\\
        D_2(w)&=(1,7)(3,4,5,6)(2)\\
        D_3(w)&=(1,2)(7,4,5,6)(3)\\
        D_4(w)&=(1,2)(3,7,5,6)(4)\\
        D_7(w)&=(1,2)(3,4,5,6)(7).
    \end{align*}
\end{example}

\begin{remark}
    For our purpose, we define the \nameref{addingswapping} from $S_{2n}$ to $S_{2n+1}$. Indeed, this mapping method can be defined from $S_n$ to $S_{n+1}$, for all $n\in \Z^+$.  
\end{remark}

\section{Prototypes}
\setstretch{1.5}

In this section, we introduce some simple lemmas that serve as prototypes for Theorem 1 and Theorem 2. 

\begin{lemma}\label{EE=OE}
    Let $n\in \Z^+$, $n>1$. Then $|EE_n| = |OE_n|$.
\end{lemma}

The simple proof of this lemma is left to the reader. A bijective proof is provided in \cite{solution}.

\begin{lemma}\label{basiclemma}
    Let $n\in \Z^+$. Then $\z|EE_{2n+1}\y| = (2n+1)\z|EE_{2n}\y|$.
\end{lemma}

This can be proved easily by realizing $EE$ and $OE$ are equally distributed in $S_n$ (\hyperref[EE=OE]{Lemma 3.1}). 

\begin{remark}\label{remark3.1}
    The lemmas above also confirm that $|OE_{2n+1}| = (2n+1)|OE_{2n}|$.
\end{remark}

\section{Proof of Theorem 1}\label{proof of theorem 1}
\setstretch{1.5}
In this section, we prove Theorem 1. 

\begin{theorem}
\label{theorem:1}
    Let $n\in \Z^+$ and $i\in [3]$. Then
    $$\mathbf{\z|EE^{(i)}_{2n+1}\y| = (2n+1)\z|EE^{(i)}_{2n}\y|}.$$
\end{theorem}

Note that: 
\begin{align*}
    \sum_{i\in[3]}\z|EE^{(i)}_{2n+1}\y| &= |EE_{2n+1}|  &\qquad &(\text{By \hyperref[2n+1_types]{Definition 2}})\\
    &= (2n+1)|EE_{2n}| &\qquad & (\text{By \hyperref[basiclemma]{Lemma 3.2}})\\
    &= (2n+1)\sum_{i\in [3]}\z|EE^{(i)}_{2n}\y|.  &\qquad &(\text{By \hyperref[2n_types]{Definition 1}})
\end{align*}

We will show the three types of $EE_n$ separately.

\subsection{Theorem 1: Type 1} \label{type1}
We will first show that $\z|\eeone_{2n+1}\y|= (2n+1)\z|\eeone_{2n}\y|$. 

\begin{proof}
    Note that all cycles in $\eeone_{2n}$ have their cycle lengths greater than $1$ (\hyperref[2n_types]{Definition 1}). 
    
    In this case, we can apply the \nameref{addingswapping}. Each permutation in $\eeone_{2n}$ gets mapped to $(2n+1)$ different permutations in $\eeone_{2n+1}$.
\end{proof}

\subsection{Theorem 1: Type 2 and Type 3}\label{theorem1type2}

\label{1} Building on top of Theorem 1: type 1 in \hyperref[type1]{Section 4.1}, we know that \hfill (1)
$$\mathbf{\z|\eetwo_{2n+1}\y|= (2n+1)\z|\eetwo_{2n}\y|} \iff \mathbf{\z|\eethree_{2n+1}\y|= (2n+1)\z|\eethree_{2n}\y|}.$$

As a result, it is sufficient to prove just one of them.

We now give a proof for $\z|EE_{2n+1}^{(i)}\y|= (2n+1)\z|EE_{2n}^{(i)}\y|$ for $i=2,3$. Additionally, we will prove one of the three types for \hyperref[theorem:2]{Theorem 2}, which is $PS^{(3)}_{2n+1} = (2n+1)PS^{(3)}_{2n}$. 

\begin{definition}\label{definition4}
    Let $n\in \Z^+$ and $\eta \in \eethree_{n}$.
    Define $A_\eta$ to be the set of all elements in odd cycles in $\eta$.
\end{definition}

\begin{definition}
    Let $n\in \Z^+$ and $A\subset [n]$. 
    Define $$EE^{(3):A}_{n} := \z\{\eta \in \eethree_{n}: A_\eta = A\y\}.$$
\end{definition}

\begin{lemma} \label{EE3A}
    If $A\subset [2n+1]$, $|A| = 2c+1$ for some $c\in \Z^+$ and $c < n$.
    Then $$EE^{(3):A}_{2n+1} \cong \z(\eeone_{2n-2c} \times \eetwo_{2c+1}\y).$$
\end{lemma}
    \begin{proof}
        Assume $A\subset [2n+1]$, $|A| = 2c+1$ for some $c\in \Z^+$ and $c < n$.
        Then \begin{align*}
            EE^{(3):A}_{2n+1}
            &= \z\{\eta \in \eethree_{2n+1}: A_\eta = A\y\}\\
            &= \z\{\eta\in S_{2n+1}: \z\{i\in [2n+1]: \text{the length of the cycle containing $i$ is odd}\y\} = A\y\}\\
            &\cong \z(\eeone_{2n-2c} \times \eetwo_{2c+1}\y).
        \end{align*}
        
    The insight behind the last line is that, instead of focusing on the partition of the entire $EE^{(3)
    }_{2n+1}$, we examine each permutation in two separate parts: the partition of even cycles and the partition of odd cycles. Referring to the three types of $EE_{2n+1}$ in \hyperref[2n+1_types]{Definition 2}, the partition of only odd or only even cycles can be found in $\eeone_{2n}$ or $\eetwo_{2n+1}$, respectively.

    \begin{example}
        Take $A = \{5,6,7\}$ and $(1,2)(3,4)(5,6,7) \in EE^{(3):A}_{7}$. 

        This permutation has a partition $(2,2,3)$. By focusing on the partitions of even and odd cycles separately, we see $(1,2)(3,4)$ as $(2,2)$, a partition of $\eeone_{4}$, and $(5,6,7)$ as $(3)$, a partition of $\eetwo_3$.
    \end{example}
    \end{proof}

\begin{proposition}\label{proposition}
    Let $n\in \Z^+$. Then 
    \begin{enumerate}
        \item  $\z|\eetwo_{2n+1}\y|= (2n+1)\z|\eetwo_{2n}\y|$
        \item $\z|\eethree_{2n+1}\y|= (2n+1)\z|\eethree_{2n}\y|$
    \end{enumerate}
\end{proposition}

We will use strong induction to show some interim results and then bijectively prove the \hyperref[proposition]{Proposition}. Furthermore, we will present a partial result for Theorem 2. 

\begin{proof}
    Base case: $\z|\eetwo_{3}\y| = 3 = 3\z|\eetwo_2\y|$.
    
    Assume $\z|EE^{(i)}_{2k+1}\y| = (2k+1)\z|EE^{(i)}_{2k}\y|,$ for all $k\in \Z^+, 1\leq k\leq n-1$ and $i = 2,3.$
    
    Fix $a\in A \subset [2n+1]$, $|A| = 2c+1, 1\leq c \leq n-1.$
    
    By induction hypothesis, 
    $(k=c,\  i=2),$ $\z|\eetwo_{2c+1}\y|= (2c+1)\z|\eetwo_{2c}\y|.$
    
    We use strong induction to show that \hfill(2) \label{(2)}$$\mathbf{\underset{{\overset{\eta \in \eethree_{2n+1}}{A_\eta = A}}}{\sum} 1 = (2c+1) \underset{{\overset{\eta \in \eethree_{2n}}{A_\eta = A \backslash \{a\}}}}{\sum} 1.}$$  

    Moreover, $$\mathbf{\underset{{\overset{\eta \in PS^{(3)}_{2n+1}}{A_\eta = A}}}{\sum} 1 = (2c+1) \underset{{\overset{\eta \in PS^{(3)}_{2n}}{A_\eta = A \backslash \{a\}}}}{\sum} 1.}$$

    For LHS: 
\begin{align*}
    \underset{\substack{\eta \in \eethree_{2n+1} \\ A_\eta = A}}{\sum} 1 &= \z|EE^{(3):A}_{2n+1}\y|\\
    &= \z|\eeone_{2n-2c} \times \eetwo_{2c+1}\y| &\qquad &(\text{By \hyperref[EE3A]{Lemma 4.1}})\\
    &= \z|\eeone_{2n-2c} \times [2c+1] \times \eetwo_{2c}\y| &\qquad &(\text{By induction hypothesis})\\
    &= (2c+1) \left| \eeone_{2n-2c} \right| \left| \eetwo_{2c} \right|.
\end{align*}

    For RHS:
    \begin{align*}
        (2c+1) \underset{{\overset{\eta \in \eethree_{2n}}{A_\eta = A \backslash \z\{a\y\}}}}{\sum} 1 &= (2c+1) \z|\z\{\eta \in S_{2n}:
             A_\eta = A \backslash \z\{a\y\}\y\} \y|\\
        &=(2c+1)\z|\eeone_{2n-2c}\y|\z|\eetwo_{2c}\y|.
    \end{align*}
   
   Note that $\eethree_{2n+1} = PS_{2n+1}^{(3)} \cup \z(NPS_{2n+1} \cap \eethree_{2n+1}\y)$. In the proof above, we only apply induction to all odd cycles in $\eethree_{2n+1}$, without loss of generality (by \hyperref[PSdependence]{Remark 1.2}): $$\underset{{\overset{\eta \in PS^{(3)}_{2n+1}}{A_\eta = A}}}{\sum} 1 = (2c+1) \underset{{\overset{\eta \in PS^{(3)}_{2n}}{A_\eta = A \backslash \{a\}}}}{\sum} 1.$$

\vspace{2em}

We will now bijectively show that:
$$\z|\eethree_{2n+1}\y|= (2n+1)\z|\eethree_{2n}\y|.$$ 

Moreover,
$$\z|PS^{(3)}_{2n+1}\y|= (2n+1)\z|PS^{(3)}_{2n}\y|.$$

For LHS: 
\begin{align*}
     \left|\eethree_{2n+1}\right| &= \sum_{\substack{A \subset [2n+1]  \\ 1 < |A| < 2n+1}} \sum_{a\in A} \sum_{\substack{\eta \in \eethree_{2n} \\ A_\eta = A \backslash \{a\}}} 1 &\qquad &(\text{By \hyperref[(2)]{(2)} })\\
    &= \sum_{a\in [2n+1]} \sum_{\substack{A \subset [2n+1] \\ 1 < |A| < 2n+1\\ a\in A}} \sum_{\substack{\eta \in \eethree_{2n} \\ A_\eta = A \backslash \{a\}}} 1 &\qquad &(\text{Change the order})\\
    &= (2n+1) \sum_{\substack{A \subset [2n+1] \\ 1 < |A| < 2n+1\\ 2n+1\in A}} \sum_{\substack{\eta \in \eethree_{2n} \\ A_\eta = A \backslash \{2n+1\}}} 1. &\qquad &(\text {By symmetry})
\end{align*}

    For RHS:
    \begin{align*}
        (2n+1)\z|\eethree_{2n}\y| = (2n+1)\sum_{\substack{B \subset [2n] \\ 1 < |B| < 2n}} \sum_{\substack{\eta \in \eethree_{2n} \\ A_\eta = B}} 1.
    \end{align*}
    
    Using change of variable such that: $B = A \backslash \{2n+1\}$ and $A = B \cup \{2n+1\}$, we know that, $\z|\eethree_{2n+1}\y|= (2n+1)\z|\eethree_{2n}\y|$, which (according to \hyperref[1]{(1)}) implies the truth of  $\z|\eetwo_{2n+1}\y|= (2n+1)\z|\eetwo_{2n}\y|$.

    The bijective proof above focuses only on the odd cycles of each permutation in $\eethree_{2n+1}$. Therefore, by \hyperref[PSdependence]{Remark 1.2}, the same method is applicable to $PS^{(3)}_{2n+1} \subset \eethree_{2n+1}$. Thus,  $\z|PS^{(3)}_{2n+1}\y|= (2n+1)\z|PS^{(3)}_{2n}\y|$.

\end{proof}

\section{Proof of Theorem 2}\label{theorem2}
\setstretch{1.5}

In this section, we prove Theorem 2. 
\begin{theorem}
    \label{theorem:2}
    Let $n\in \Z^+$ and $i\in [3]$. Then
    $$\mathbf{\z|PS^{(i)}_{2n+1}\y| = (2n+1)\z|PS^{(i)}_{2n}\y|}.$$
\end{theorem}

The proof of this theorem is very similar to the \nameref{proof of theorem 1}. We will prove the three types individually.  

\subsection{Theorem 2: Type 1}
We will show that $\z|PS^{(1)}_{2n+1}\y| = (2n+1)\z|PS^{(1)}_{2n}\y|$.

\begin{proof}
    Note that $PS_n^{(1)}\subset \eeone_n$ (by \hyperref[definition3]{Definition 3}), applying the same method, \nameref{addingswapping} presented in \hyperref[type1]{Section 4.1}, we confirm that $\z|PS^{(1)}_{2n+1}\y| = (2n+1)\z|PS^{(1)}_{2n}\y|$.
\end{proof}

\subsection{Theorem 2: Type 2}
We will show that $\z|PS^{(2)}_{2n+1}\y| = (2n+1)\z|PS^{(2)}_{2n}\y|$. 

Instead of proving this part of Theorem 2 directly, we will demonstrate an equivalent relationship, which once established, will imply the truth of $\z|PS^{(2)}_{2n+1}\y| = (2n+1)\z|PS^{(2)}_{2n}\y|$.
\begin{proof}  
    We will show that $PS_n^{(2)} = \eetwo_n$. 

    \begin{enumerate}
        \item $PS_n^{(2)} \subset \eetwo_n$: By \hyperref[definition3]{Definition 3}, $PS_n^{(2)}\subset \eetwo_n$.
        \item $\eetwo_n \subset PS_n^{(2)}$: By \hyperref[oddPS]{Remark 1.1}, it is clear that every odd cycle is a perfect square. Since every permutation in $\eetwo_n$ is a product of \textbf{all odd} cycles (recall \hyperref[2n_types]{Definition 1} and \hyperref[2n+1_types]{Definition 2}), every permutation in $\eetwo_n$ is also in $PS_n^{(2)}$. Thus, $\eetwo_n \subset PS^{(2)}_n$.
    \end{enumerate}
    Now that we use the result from \hyperref[theorem:1]{Theorem 1}, we prove that $$\z|PS^{(2)}_{2n+1}\y| = \z|\eetwo_{2n+1}\y| = (2n+1)\z|\eetwo_{2n}\y| = (2n+1)\z|PS^{(2)}_{2n}\y|.$$

\end{proof}

\subsection{Theorem 2: Type 3}

The proof of $\z|PS^{(3)}_{2n+1}\y|= (2n+1)\z|PS^{(3)}_{2n}\y|$ has been done in \hyperref[theorem1type2]{Section 4.2}. 

\vspace{2em}

\subsection{Corollary}
In this section, we prove the corollary.  

\begin{corollary}
    $\mathbf{\alpha(2n+1) = (2n+1)\alpha(2n)}$.
\end{corollary}

\begin{proof}
    At the beginning of this paper, we state that $\alpha(n)$ represents the number of perfect square permutations in $S_n$. Thus, we have $\alpha(n) = \underset{i \in [3]}{\sum} \z|PS_n^{(i)}\y|$.

    Therefore, by \hyperref[theorem:2]{Theorem 2}, 
    \begin{align*}
        \alpha(2n+1) = \sum_{i\in [3]}\z|PS^{(i)}_{2n+1}\y| = (2n+1)\sum_{i \in [3]}\z|PS^{(i)}_{2n}\y| = (2n+1)\alpha(2n).
    \end{align*}
\end{proof}

\section{Additional Exploration of Permutations with Odd Number of Even Cycles }
\setstretch{1.5}

After presenting our main results that are related to $EE_n$, we now explore the possible types in $OE_n$. Similar to \hyperref[theorem:1]{Theorem 1}, we will show that there are three distinct types for both $OE_{2n}$ and $OE_{2n+1}$, which correspond to each other. 

\begin{definition}
    Let $n\in \Z^+$. We divide $OE_{2n}$ into three \textbf{distinct} types:
    \begin{itemize}
        \item Type 1, denoted $OE_{2n}^{(1)}:=\{w\in OE_{2n}: w$ only has even cycles\}.
        \item Type 2, denoted $OE_{2n}^{(2)}:=\{w\in OE_{2n}:w$ has no $1$-cycle\} $\backslash OE_{2n}^{(1)}$.
        \item Type 3, denoted $OE_{2n}^{(3)}:=\{w\in OE_{2n}:w$ has at least one $1$-cycle\}. 
    \end{itemize}
    Note that: 
    $$\z|OE_{2n}\y| = \underset{i\in [3]}{\sum} \z|OE_{2n}^{(i)}\y|.$$
\end{definition}

\vspace{1em}

\begin{definition}
    Let $n\in \Z^+$. We can divide $OE_{2n+1}$ into three \textbf{distinct} types:
    \begin{itemize}
        \item Type 1, denoted $OE_{2n+1}^{(1)}:=\{w\in OE_{2n+1}:w$ only has even cycles and one $1$-cycle\}.
        \item Type 2, denoted $OE_{2n+1}^{(2)}:=\{w\in OE_{2n+1}:w$ only has one $1$-cycle\} $\backslash OE_{2n+1}^{(1)}$.
        \item Type 3, denoted $OE_{2n+1}^{(3)}:= OE_{2n+1}$ $\backslash \z(OE_{2n+1}^{(2)} \cup OE_{2n+1}^{(1)}\y)$. 
    \end{itemize}
    Note that: 
    $$\z|OE_{2n+1}\y| = \underset{i\in [3]}{\sum} \z|OE_{2n+1}^{(i)}\y|.$$
\end{definition}

\vspace{1em}

We will show that $\z|OE_{2n+1}^{(i)}\y| = (2n+1)\z|OE_{2n}^{(i)}\y|$, for all $n\in \Z^+$ and $i \in [3]$.

\begin{proof}
    By applying the \hyperref[addingswapping]{Adding and Swapping Mapping Method}, we can effortlessly confirm that $$\z|OE_{2n+1}^{(i)}\y| = (2n+1)\z|OE_{2n}^{(i)}\y|\text{, for } i = 1, 2.$$
    \noindent By \hyperref[remark3.1]{Remark 3.1}, we confirm that $|OE_{2n+1}| = (2n+1)|OE_{2n}|$. Therefore, $$\z|OE_{2n+1}^{(3)}\y| = (2n+1) \z|EE_{2n}^{(3)}\y|.$$

\end{proof}

\section{Further Remarks}
\setstretch{1.5}
We hope the methods demonstrated in this paper will be beneficial for future research in this area. We are particularly interested in identifying more potential types within $S_{2n}$ and $S_{2n+1}$ that may provide special formulations similar to those proposed in \hyperref[theorem:1]{Theorem 1} and \hyperref[theorem:2]{Theorem 2}. Additionally, it will be very interesting to see if our methods can be applied when dealing with higher powers.

\section{Acknowledgements}
\setstretch{1.5}
 I wish to extend my sincere appreciation to my supervisor, Jiyuan (Maki) Lu, for his support, guidance, and mentorship throughout this research. His insightful feedback was crucial to the completion of this paper. I also want to express my gratitude to Professor Richard P. Stanley for the formulation that is provided. Furthermore, I am grateful to my research group member, Parham Tayyebi, for the insightful discussions that have enriched this work.

\printbibliography

@article{Blum,
   author = {Joseph Blum},
   title = {Enumeration of the square permutations in $S_n$},
   journal = {Journal of Combinatorial Theory, Series A},
   volume = {17},
   pages = {156-161},
   year = {1974},
   doi = {10.1016/0097-3165(74)90002-8}
}

@book{oeis,
  author = {N. J. A. Sloane and Simon Plouffe},
  title = {The Encyclopedia of Integer Sequences},
  publisher = {Academic Press},
  year = {1995},
  note = {https://oeis.org/A003483}
}

@online{problem,
  author    = "Richard P. Stanley",
  title     = "BIJECTIVE PROOF PROBLEMS",
  year      = "2009",
  url       = "https://math.mit.edu/~rstan/bij.pdf",
  note      = "[Online; accessed May 2024]",
}

@article{solution,
   author = {Stoyan Dimitrov and Luz Grisales and Rodrigo Pasada and Michael Schleppy},
   title = {BIJECTIVE PROOF PROBLEMS - SOLUTIONS},
   pages = {1-41},
   doi = {https://stoyandimitrov.net/bijSol.pdf}
}
\end{document}